\documentclass[10pt,a4,twoside]{amsart}

\usepackage{amsmath,amsthm,amsfonts,amssymb}

\usepackage{mathtools} 
\usepackage{mathrsfs} 

\usepackage[active]{srcltx} 
\usepackage[pdftex,bookmarks=true]{hyperref} 

\usepackage{graphicx}

\usepackage[alphabetic]{amsrefs}

\usepackage{float}

\newcommand{\Z}{\mathbb{Z}}

\numberwithin{equation}{section}

\theoremstyle{plain} 
\newtheorem{thm}[equation]{Theorem}

\newtheorem{lem}[equation]{Lemma}
\newtheorem{prop}[equation]{Proposition}

\theoremstyle{definition}
\newtheorem{defn}[equation]{Definition}

\theoremstyle{remark}
\newtheorem{rem}[equation]{Remark}

\title{The invariance of knot lattice homology}
\author{Matthew Jackson}

\begin{document}

\begin{abstract}
Assume $\Gamma$ is a negative-definite forest with exactly one unframed vertex, and $M(\Gamma)$ is the resulting plumbed 3-manifold with a knot embedded. We show that the filtered lattice chain homotopy type of $\Gamma$ is an invariant of the diffeomorphism type of $M(\Gamma)$.
\end{abstract}

\maketitle

\section*{Introduction}

Lattice (co)homology is a theory introduced by András Némethi \cite{nemethi} for negative definite plumbed 3-manifolds and is conjecturally isomorphic to Heegaard Floer homology. Oszváth, Stipsicz and Szabó \cite{knotsinlattice} defined a knot refinement for lattice homology which is a filtration on the lattice chain complex.  It is natural to ask whether this filtration is in fact an invariant of the knot (and thus if the ``knot filtration" is well defined).

We now recall the setup. Let $G$ be a negative definite framed forest. To define the resulting manifold $M(G)$, see the graph as a link in the boundary of $\mathbb{D}^4$ (by replacing vertices by loops and edges by crossings) and plumb according to the framing on each vertex. This defines a new 4-manifold whose boundary is the 3-manifold $M(G)$. Similarly in the knot lattice setup we let $\Gamma$ be a negative definite plumbed forest with exactly one unframed vertex $v_0$ and let $G:=\Gamma - \{v_0\}$ . The same construction yields $M(G)$ and a knot $K_{v_0}$ embedded in it, thus we will often refer to $M(\Gamma) := (M(G),K_{v_0})$. Oszváth, Stipsicz and Szabó \cite{knotsinlattice} define the knot filtration combinatorially using the graph $\Gamma$.

Our main result is:

\begin{thm}\label{thm1}Let $\Gamma$ be a negative definite forest with exactly one unframed vertex. The filtered lattice chain homotopy type of $\Gamma$ is an invariant of the oriented diffeomorphism type of $M(\Gamma)$.
\end{thm}

To prove it, we will show the following propositions:

\begin{prop}\label{prop11}
Let $\Gamma_1$ and $\Gamma_2$ be two negative definite forests with exactly one unframed vertex each, such that $M(\Gamma_1)$ and $M(\Gamma_2)$ are diffeomorphic. Then, blow-ups and blow-downs are sufficient to turn $\Gamma_1$ into $\Gamma_2$. 
\end{prop}
 
\begin{prop}\label{prop02}
Let $\Gamma_1$ and $\Gamma_2$ be two negative definite forests that differ by a blow-up. Then the filtered lattice chain complexes $(\mathbb{CF}^-(G_1),A)$ and $(\mathbb{CF}^-(G_2),A')$ and chain homotopic.
\end{prop}

We will prove Proposition \ref{prop11} in Section 1 and Proposition \ref{prop02} in Sections 2 and 3. 

This topic was suggested by András Stipsicz for an internship I did under his supervision at the Alfréd Rényi Institute in Budapest. I would like to thank him for his precious help.

\section{Sufficiency of blow-ups and blow-downs}

First some preliminary definitions:

\begin{defn}Let $\Gamma$ be a graph with framings on all its vertices except for $v_0$ and $G$ the graph obtained from $\Gamma$ by deleting $v_0$ and all edges adjacent to it. $\Gamma$ is called \emph{negative definite} if the intersection matrix of $G$ (with the framings on the diagonal) is negative definite. 
\end{defn}

We refer to \cite{neumann} for the definitions of blow-down and blow-up (operation R1 and its inverse) and adapt them to graphs with an unframed vertex by considering it as having framing $-\infty$: the only forbidden operation is to blow-down the unframed vertex.
We only allow blow-ups or blow-downs that stay in the class of graphs considered (negative definite trees with exactly one unframed vertex). Therefore there are three types of blow-ups (and blow-downs) to study : generic blow-up, the blow-up of a vertex and the blow-up of an edge.  

\begin{defn}
Two negative definite forests with exactly one unframed vertex $\Gamma$ and $\Gamma'$ are said to be \emph{equivalent} if they are related by a finite sequence of blow-ups and blow-downs. This defines an equivalence relation on the set of negative definite graphs with exactly one unframed vertex.  Denote their equivalence class by $[\Gamma]$.
\end{defn}
\begin{rem}
Notice that if two trees $\Gamma_1$ and $\Gamma_2$ are equivalent, then they define diffeomorphic 3-manifolds and knots $M(\Gamma_1)\cong M(\Gamma_2)$.
\end{rem}

\begin{defn}A graph $\Gamma$ with exactly one unframed vertex is said to be \textit{reduced} if no blow-downs can be applied to $\Gamma$.
\end{defn}

\begin{rem}Let $\Gamma'$ denote the graph $\Gamma$ after a blow-down. A simple exercise in linear algebra shows that if $\Gamma$ is negative definite, then so is $\Gamma'$. This implies that any class $[\Gamma]$ contains reduced graphs.
\end{rem}

\begin{prop}\label{prop1}
Let $\Gamma$ and $\Gamma'$ be negative definite trees with exactly one unframed vertex.  
\begin{enumerate}
\item The equivalence class $[\Gamma]$ contains a unique reduced graph.
\item Suppose $\Gamma_1$ and $\Gamma_2$ are in reduced form and $M(\Gamma_1)\cong M(\Gamma_2)$. Then $\Gamma_1 \cong \Gamma_2$.
\end{enumerate}
\end{prop}

\begin{proof} If $\Gamma$ is reduced, let $\Gamma_{norm}=\Gamma_{norm}(n)$ denote the graph $\Gamma$ with the following changes : choose $-n$ large enough such that $\Gamma(n)$ satisfies conditions N1, N2, N4, N5, and N6 of normal form (see \cite{neumann}) and $n$ is the smallest framing on $\Gamma(n)$. On any component of $\Gamma(n)$ with same shape as figure \ref{nochange} ($k\geqslant 0$ and $e\leqslant -1$) don't apply any changes.
\begin{figure}[H]
	\begin{tabular}{cc}
		\includegraphics[width=0.5\linewidth]{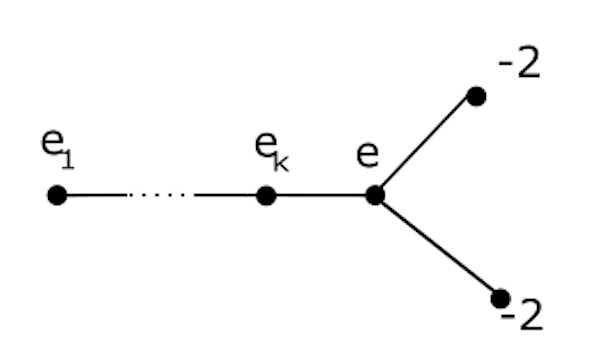}
		\includegraphics[width=0.5\linewidth]{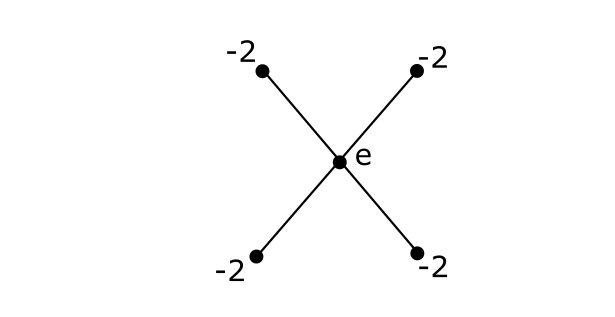}
	\end{tabular}
	\caption{\label{nochange} These components don't change
	} 
\end{figure}
On all other components,
if $e\leqslant -3$ or $e$ has at least 2 edges incident to it and $m\geqslant 0$, (we use the writing conventions from \cite{neumann})
\begin{figure}[H]
	\begin{tabular}{cc}
		\includegraphics[width=0.5\linewidth]{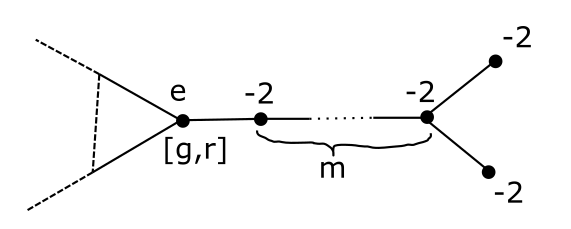}
		\includegraphics[width=0.5\linewidth]{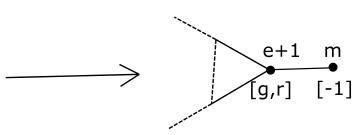}
	\end{tabular}
	\caption{\label{change1} replace the left figure with the right figure
	} 
\end{figure}
If $m\leqslant -1$,
\begin{figure}[H]
	\begin{tabular}{cc}
		\includegraphics[width=0.5\linewidth]{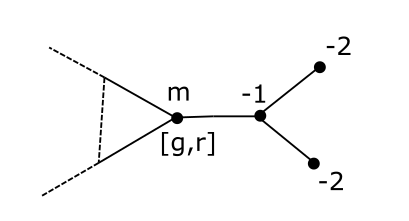}
		\includegraphics[width=0.5\linewidth]{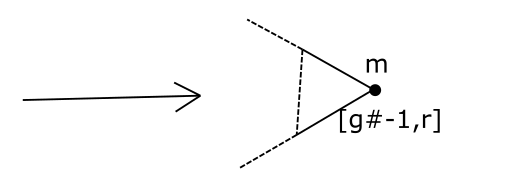}
	\end{tabular}
	\caption{\label{change2} replace the left figure with the right figure
	} 
\end{figure}
These changes make $\Gamma_{norm}$ satisfy condition N3 and respect the diffeomorphism type of $\Gamma(n)$.  Therefore, $\Gamma_{norm}$ is the normal form of $\Gamma(n)$.  Now suppose $\Gamma'$ is equivalent to $\Gamma$ and reduced. This implies that for $-m$ large enough $\Gamma'(m)$ also has normal form $\Gamma_{norm}$ (see \cite{neumann} Theorem 4.2, one might need to increase $-n$). Since $n$ is the smallest framing of $\Gamma_{norm}$, it is possible to identify the unframed vertices of $\Gamma$ and $\Gamma'$,  and figures \ref{change1} and  \ref{change2}  define an injective map, therefore $\Gamma' \cong \Gamma$.   This shows (1). (2) is a consequence of \cite{neumann} Theorem 4.2,  of the injectivity of the above construction and of the fact that $\Gamma_{norm}$ is in normal form implies $\Gamma(n)$ is reduced.
\end{proof}

Proposition \ref{prop11} is now a direct consequence of Proposition \ref{prop1}. 
\section{Blow-up and blow-down of a vertex}

Let $\Gamma$ be a negative definite tree with an unframed vertex $v_0$ and 
\(G:= \Gamma - \{v_0\}\).
Notice that we can consider each component of $\Gamma -\{v_0\}$ independently and then apply the connected sum formula (\cite{knotsinlattice}, Theorem 4.8) in order to check the invariance of the filtered lattice chain homotopy type. Therefore we can suppose $v_0$ is a leaf. Let $v$ be a framed vertex of $\Gamma$ with framing $m_v$.  $\Gamma'$ will denote the graph obtained by blowing up $v$.
Thus $v\in G':=\Gamma' - \{v_0\}$ has framing $m_v -1$, $G'$ has a new vertex $w$ with framing $-1$, and $v$ and $w$ are connected by an edge. Let $\mathbb{CF^-}(G)$ denote the lattice chain complex of $G$ and let $A$ denote the knot filtration.  As in \cite{knotsinlattice} we define the blow-down map:
\[
\begin{array}{cccl}
P : & \mathbb{CF^-}(G') & \to & \mathbb{CF^-}(G)\\
& [(K,p,j),E] & \mapsto & 
\begin{cases}
U^s[(K,p+j),E]& \text{if } w\notin E\\
0             & \text{if } w\in E
\end{cases}
\end{array}
\]
with 
\[
s := g[(K,p+j),E] - g[(K,p,j),E] + \frac{j^2 -1}{8}
\]
and the blow-up map:
\[
\begin{array}{cccl}
R : & \mathbb{CF^-}(G) & \to & \mathbb{CF^-}(G')\\
& [(K,p),E] &\mapsto & C_w([K,p+1,-1),E]
\end{array}
\]
The appendix of \cite{knotsinlattice} (on arXiv) shows that $P$ and $R$ are both graded chain maps and that they are graded homotopy equivalences. Let's check they respect the knot filtrations.

\begin{lem}
$P$ is a filtered chain map
\end{lem}
\begin{proof}
We already know from \cite{knotsinlattice} Lemma 7.10 that $P$ is a chain map. Let's prove it respects the knot filtrations. Let $\Sigma \in H_2(X_G;\mathbb{Q})$ be the homology class satisfying:
\[ \Sigma = v_0 + \sum_{j=1}^{n}a_j.v_j + a_v.v\ \ \  \text{and }\ \ \  v_i.\Sigma = 0 \ \forall i\in \{1,…,n\}
\]
where $X_G$ is the plumbed 4-manifold defined by $G$. $\Sigma$ exists and is unique because $G$ is assumed to be negative definite. Similarly a quick calculation gives $\Sigma' \in H_2(X_{G'};\mathbb{Q})$: 
\[
\Sigma' = v_0 + \sum_{j=1}^{n}a_j.v_j + a_v.v + a_v.w
\]
Therefore,
\begin{align*}
\Sigma^2 &= v_0^2 + \left(\sum_{j=1}^{n}a_j.v_j\right)^2 + a_v^2m_v + 2v_0.\sum_{j=1}^{n}a_j.v_j + 2a_vv_0.v + 2a_vv.\sum_{j=1}^{n}a_j.v_j\\
& = \Sigma'^2
\end{align*}
Recall that 
\[A(U^j[K,E]) =-j +  \frac{1}{2}\left( L_{[K,E]}(\Sigma) + \Sigma^2\right)\]
As in \cite{knotsinlattice} the set of characteristic elements is $\mathrm{Char}(G) := \{L:H_2(X_G, \Z) \to \Z \mid K(x) =  x.x\  \mathrm{(mod }\ 2)\}$. In the following $(K,p)\in \mathrm{Char}(G)$ will denote the characteristic element worth  $p \in \Z$ on $v$ and whose restriction to $G-\{v\}$ is $K$.  Similarly $(K,p,j)\in \mathrm{Char}(G')$ will denote the characteristic element worth $p$ on $v$, $j$ on $w$ and whose restriction to  $G'-\{v,w\}$ is $K$. We calculate,
\begin{align*}
L_{[(K,p,j),E]}(\Sigma) &= L_{[(K,p,j),E]}(v_0) + K(\sum_{j=1}^{n}a_j.v_j) + (p+j)a_v\\
L_{[(K,p+j),E]}(\Sigma')& = L_{[(K,p+j),E]}(v_0) + K(\sum_{j=1}^{n}a_j.v_j) + (p+j)a_v
\end{align*}
and 
\begin{align*}
L_{[(K,p,j),E]}(v_0) &= -v_0^2 + 2g[(K,p,j),E] - 2g[(K,p,j) + 2v_0^*,E]\\ 
L_{[(K,p+j),E]}(v_0) &= -v_0^2 + 2g[(K,p+j),E] - 2g[(K,p+j) + 2v_0^*,E]
\end{align*}
Therefore,  if $w\notin E$,
\[ 
A(P[(K,p,j),E]) - A'[(K,p,j),E] = -\frac{j^2-1}{8} + g[(K,p,j) + 2v_0^*,E] - g[(K,p+j) + 2v_0^*,E]
\]
Notice that if $v\notin I \subset E$ and $w\notin E$, 
\begin{align*}
2f[(K,p,j)+2v_0^*, I\cup v] &= 2f[(K,p+j)+2v_0^*, I\cup v] - j - 1\\
2f[(K,p,j)+2v_0^*, I] &= 2f[(K,p+j)+2v_0^*, I]
\end{align*} 
Therefore, 
\[
g[(K,p,j) + 2v_0^*, E] - g[(K,p,j)+2v_0^*, E] \leqslant \max(0, \frac{-j-1}{2})
\]
Since $(K,p,j)$ is a characteristic element and $w^2 = -1$, $j$ must be odd. Therefore, $\frac{-j-1}{2}-\frac{j^2-1}{8} \leqslant 0$. We conclude:
\[
A(P[(K,p,j),E]) - A'[(K,p,j),E] \leqslant 0
\]
\end{proof}

As in \cite{knotsinlattice} Definition 7.5,  a generator $[K,E] \in \mathbb{CF^-}(G')$ (with $w\notin E$) is of type-a if $a_w[K+2i_0w^*, E\cup w] = 0$ and of type-b otherwise. Notice that $w$ is a good vertex and $-w^2 = 1$. Let $T = T_{[K,E]}$ be equal to $1$ if $[K,E]$ is of type-a and $-1$ if $[K,E]$ is of type-b.
Now consider the map $H_0 : \mathbb{CF^-}(G') \to \mathbb{CF^-}(G')$ defined as : 
\[
H_0[K,E] =
\begin{cases}
0 & \text{if } w\in E \ \text{or } (T-2) \leqslant K(w) < T\\
[K,E\cup w] & \text{if } w\notin E \ \text{and } K(w) \geqslant T\\
[K-2w^*, E\cup w] & \text{if } w\notin E \ \text{and } K(w)< (T-2)
\end{cases}
\]
As in \cite{knotsinlattice}, $H_0$ increases the Maslov grading by 1. Define \(C_0 = \mathrm{Id} + \partial \circ H_0 + H_0 \circ \partial\).  For each $[K,E]$ there is an $N = N_{[K,E]}$ such that the $N$-th iterate of $C_0$ stabilizes. Denote $C_w := C_0^\infty$. By \cite{knotsinlattice} Theorem 7.8, 
\begin{equation}\label{eq1}
C_w[K,E] = 0\ \text{if } w\in E
\end{equation}
and 
\begin{equation}\label{eq2}
C_w[K,E] = [K,E] + \partial \circ H_w[K,E] + H_w\circ \partial [K,E]
\end{equation}
By \cite{knotsinlattice} Example 7.9, 
\begin{equation}\label{eq3}
C_w[(K,p,j),E] = C_w[(K,p+j+1,-1),E]
\end{equation}

\begin{lem}
$H_0$ is filtered.
\end{lem}
\begin{proof}

For $[K,E]$ a generator of $\mathbb{CF^-}(G')$ with $w\notin E$, consider $H_0[K,E]$.
\begin{itemize}
\item If $K(w) \geqslant 1$, then by \cite{knotsinlattice} Lemma 7.3, \(a_w[K,E\cup w] = 0\). Similarly, \((K+2v_0^*)(v) = K(v) \geqslant 0$, therefore \(a_w[K+ 2v_0^*,E\cup w] = 0\). Thus, by \cite{knotsinlattice} formula (3.2),  \(A'[K,E\cup w] - A'[K,E] = 0\).
\item If $K(w) \leqslant -3$ then $(K-2w^*+2v_0^*)(w) = K(w)+2 \leqslant -1$. Therefore, by \cite{knotsinlattice} Lemma 7.3 and formula (3.3),  \(A'[K-2w^*,E\cup v] - A'[K, E] = 0\).
\item As $K$ is a characteristic element, $K(w)$ must be odd, therefore $K(w) \neq 0$. Suppose now that $K(w) = -1$. If \([K,E]\) is of type-b then $H_0[K,E] = 0$.  If \([K,E]\) is of type-a, then $H_0[K,E] = [K,E\cup w]$ and 
\[
a_w[K,E\cup w] = g[K,E] - g[K,E\cup w] = 0
\]
Therefore, by \cite{knotsinlattice} formula (3.2), $A'[K,E] \geqslant A'[K,E\cup w]$ is equivalent to \[a_w[K+2v_0^*,E\cup w] := g[K+2v_0^*, E] - g[K+2v_0^*, E\cup w] = 0\]
By minimality, $g[K+2v_0^*, E\cup w] \leqslant g[K+2v_0^*, E]$, and since $v_0$ is a leaf, we have
\begin{equation}\label{eq4}
g[K,E] \leqslant g[K + 2v_0^*, E\cup w] \leqslant g[K+2v_0^*,E] \leqslant g[K,E] +1 
\end{equation}
There are a certain number of subsets $I\subset E$ such that $g[K,E] = f[K,I]$. First,  suppose that there is such a subset $I_0$ satisfying $v_0.\sum_{u\in I_0}u = 0$. Then
\begin{align*}
2f[K+2v_0^*, I_0] &= K(\sum_{u\in I_0}u) + 2v_0.\sum_{u\in I_0}u + (\sum_{u\in I_0}u)^2\\
&= 2f[K,I]
\end{align*}
Therefore, $g[K+2v_0^*, E] = g[K,E] = g[K+2v_0^*, E\cup w]$.  Now suppose that 
\begin{equation}\label{eq5}
g[K,E] = f[K,I] \Rightarrow v_0.\sum_{u\in I}u = 1
\end{equation} 
This implies that $f[K+2v_0^*, I] = f[K,I] + 1$ for any such $I$. If $f[K+2v_0^*, I] = g[K+2v_0^*; E\cup w]$ then $g[K+2v_0^*; E\cup w] = g[K+2v_0^*; E]$ by (\ref{eq4}).  If $g[K+2v_0^*; E\cup w] = f[K+2v_0^*, I] -1 = f[K,I]$,  then there is a subset $J \subset E\cup w$ verifying $f[K+2v_0^*, J] = g[K+2v_0^*, E\cup w]$. Since $f[K+2v_0^*, J] \geqslant f[K,J]$, $f[K,J]$ and $f[K,I]$ must be equal.  Moreover, the fact that $K(w) = -1$ guarantees that $f[K,J] = f[K,J - w]$, therefore $v_0.\sum_{u\in J-w}u = 1$ by (\ref{eq5}), thus $f[K+2v_0^*,J] = f[K,I] +1$ which is absurd because $f[K+2v_0^*,J] = g[K+2v_0^*, E\cup w] =  f[K,I]$.
\end{itemize}
\end{proof}

As an immediate corollary, $C_w$ is a filtered chain map and $H_w$ is filtered.

\begin{lem}
$R$ is a filtered chain map
\end{lem}
\begin{proof}
We know $R$ is a chain map and $C_w$ is filtered. Therefore, it is sufficient to show that $A[(K,p), E] - A'[(K,p+1,-1), E]\geqslant 0$. By definition:
\[
\begin{array}{ccl}
A[(K,p), E] - A'[(K,p+1,-1), E] &=& \frac{1}{2}\left(L_{[(K,p), E]}(v_0) - L_{[(K,p+1,-1), E]}(v_0)\right)\\
&=& g[(K,p), E] - g[(K,p) + 2v_0^*, E] \\
&&+ g[(K,p+1,-1) + 2v_0^*, E]  - g[(K,p+1,-1),E]
\end{array}
\]
If $w\in E$ then $R[(K,p),E] = C_w[(K,p+1,-1),E] = 0$. Suppose $w\notin E$.
Notice that for $I\subset E$, such that $v\notin I$,
\begin{align*}
2f[(K,p+1,-1), I] & = 2f[(K,p),I]\\
2f[(K,p+1,-1), I\cup v] &= 2f[(K,p), I\cup v]\\
2f[(K,p+1,-1) + 2v_0^*, I] &= 2f[(K,p) + 2v_0^*, I]\\
2f[(K,p+1,-1) + 2v_0^*, I\cup v] &= 2f[(K,p) + 2v_0^*, I\cup v]\\
\end{align*}
Thus
\begin{align*}
g[(K,p), E] &= g[(K,p+1,-1),E]\\
g[(K,p) + 2v_0^*, E]  &= g[(K,p+1,-1) + 2v_0^*, E]
\end{align*}
Therefore,  $A[(K,p), E] - A'[(K,p+1,-1), E] = 0$
\end{proof}

\begin{prop}
$P$ and $R$ are filtered graded chain homotopy equivalences
\end{prop}
\begin{proof}
$P$ and $R$ are filtered graded chain maps.  Using the fact that $P$ is a chain map that preserves the Maslov grading and that $P\circ H_w = 0$,
\begin{align*}
P\circ R[(K,p), E] &= P[(K,p+1,-1), E] + P\circ \partial \circ H_w[(K,p+1,-1), E] + P\circ H_w \circ \partial[(K,p+1,-1), E]\\
&= U^0[(K,p),E] + 0 + 0
\end{align*}
Therefore, $P\circ R = \mathrm{Id}_{\mathbb{CF^-}(G)}$.

Let's show that $R\circ P = C_w$. If $w\in E$ both equal $0$. If $w\notin E$, then using equation (\ref{eq3}) and the fact that $P$, $R$ and $C_w$ preserve Maslov grading,
\begin{align*}
R\circ P[(K,p,j),E] = C_w[(K,p+j+1,-1),E] = C_w[(K,p,j),E]
\end{align*}
And $C_w$ is filtered chain homotopic to $\mathrm{Id}_{\mathbb{CF^-}(G')}$ (because $H_w$ is filtered).
\end{proof}

\section{Blow-up and blow-down of an edge}

Let $\Gamma$ be a negative definite tree with an unframed vertex $v_0$ and 
\(G:= \Gamma - \{v_0\}\)
By the connected sum formula (\cite{knotsinlattice}, Theorem 4.8) we can suppose $v_0$ is a leaf.  Let $v$ and $w$ be two vertices of $G$ connected by an edge with respective framings $m_v$ and $m_w$. Notice that $v$ or $w$ can be connected to $v_0$, and that they can be good or bad.  Let $\Gamma'$ denote the tree obtained by blowing-up the edge between $v$ and $w$, and suppose $G' := \Gamma' - \{ v_0\}$ is negative definite.  $\Gamma'$ has a new vertex $e$ with framing $-1$ connected to the vertices $v$ and $w$ with framings $m_v-1$ and $m_w-1$.  $\mathbb{CF^-}(G)$ will denote the lattice chain complex of $G$ and $A$ will denote the knot filtration.  

As in \cite{knotsinlattice}, consider the map $H_0$ defined as :
\[
\begin{array}{cccl}
H_0 : & \mathbb{CF^-}(G') & \to & \mathbb{CF^-}(G')\\
& [K,E] & \mapsto &
\begin{cases}
0 & \text{if } e \in E \ \text{or } K(e) = T-2\\
[K,E\cup e] & \text{if } e\notin E \ \text{and } K(e) \geqslant T\\
[K -2e^*, E\cup e] & \text{if } e\notin E \ \text{and } K(e) \leqslant T-4
\end{cases}
\end{array}
\]
where $T= -1$ if $[K,E]$ is of type-a and $T=1$ otherwise (see \cite{knotsinlattice} Definition 7.5).

\begin{lem}\label{lem1} If $e\notin E$ and $K(e) \leqslant -3$, then $b_e[K,E\cup e] = 0$.
\end{lem}
\begin{proof}
Equivalently, we will show that $A_e[K,E\cup e] \geqslant B_e[K,E\cup e]$. Recall that:
\begin{align*}
A_e[K,E\cup e] &:= \min\{f[K,I]\mid I\subset E\}\\
B_e[K,E\cup e] &:= \min \{f[K,I\cup e] \mid I \subset E\}
\end{align*}
By definition,
\[
2f[K,I\cup e] = 2f[K,I] + K(e) + e^2 + 2e.\sum_{u\in I}u
\]
and
\[
K(e) + e^2 + 2e.\sum_{u\in I}u \leqslant -3  -1 + 2\times 2 \leqslant 0
\]
Therefore, $B_e[K,E\cup e]\leqslant A_e[K,E\cup e]$.

\end{proof}

\begin{lem}\label{lem2}$H_0$ increases Maslov grading by 1.
\end{lem}
\begin{proof}
Let $\partial_u$ denote the terms of the differential where the vertex $u$ is deleted. We will check that in $\partial_e \circ H_0[K,E]$ the term $[K,E]$ always has exponent $0$, and since $\partial$ decreases the Maslov grading by 1, $H_0$ must increase it by 1.  We only need to consider the case where $e\notin E$. 
\begin{itemize}
\item If $K(e) \geqslant 1$ then by \cite{knotsinlattice} Lemma 7.3, $a_e[K,E\cup e] = 0$
\item If $K(e) \leqslant -5$ then $K(e) -2e^2 \leqslant -3$. Therefore, by Lemma \ref{lem1}, \\$b_e[K-2e^*, E\cup e] = 0$
\item If $K(e) = -1$ and $[K,E]$ is of type-a, then $i_0 = 0$. Therefore, $a_e[K,E\cup e] = 0$
\item If $K(e) = -3$ and $[K,E]$ is of type-b, ie $a_e[K+2i_0e^*, E\cup e] > 0$, then $i_0 = -1$. Therefore $b_e[K+2i_0e^*, E\cup e] = b_e[K-2e^*, E\cup e] = 0$.
\end{itemize}
\end{proof}

As in \cite{knotsinlattice}, we now define 
\[
\begin{array}{cccl}
H : & \mathbb{CF^-}(G') & \to & \mathbb{CF^-}(G')\\
& [K,E] & \mapsto & 
\begin{cases}
0 & \text{if } e \in E \ \text{or } K(e) = T-2\\
\sum_{i=0}^t U^{s_i}[K+2ie^*, E\cup e] & \text{if } e\notin E \ \text{and } K(e) = T + 2t \ \text{with } t\geqslant 0\\
\sum_{i=0}^{-t-2}U^{r_i}[K-2(i+1)e^*, E\cup e] & \text{if } e\notin E \ \text{and } K(e) = T + 2t \ \text{with } t\leqslant -2
\end{cases}

\end{array}
\]
Define $C := \mathrm{Id}_{\mathbb{CF^-}(G')} + \partial \circ H + H \circ \partial$ and $C_0 := \mathrm{Id}_{\mathbb{CF^-}(G')} + \partial \circ H_0 + H_0 \circ \partial$.  Notice that for any $[K,E]$,  $C_0^n[K,E]$ ($C_0$ composed with itself $n$ times applied to $[K,E]$) eventually stabilizes for $n$ big enough and $C_0^\infty = C$. Similarly, there is an integer $k = k_{[K,E]}$ such that $(H_0 \circ \partial)^k \circ H_0[K,E] = H[K,E]$.

$(K,p,q,l)$ will denote the characteristic cohomology class taking values $p$ on $v$, $q$ on $w$, $l$ on $e$,  and whose restriction to $G'-\{v,w,e\}$ is $K$. 

\begin{lem}\label{lem3}$C[(K,p,q,l), E\cup e] = 0$
\end{lem}
\begin{proof}
\begin{align*}
C[L,E\cup e] &= [L,E\cup e] + \partial\circ H[L,E\cup e] + H\circ \partial [L, E\cup e]\\
&= [L,E\cup e] + 0 + U^xH[L,E] + U^yH[L+2e^*, E]
\end{align*}
Using the fact that $H$ increases the Maslov grading by exactly 1, all terms in $U^xH[L,E] + U^yH[L+2e^*, E]$ will cancel except for $[L,E\cup e]$. 

\end{proof}

\begin{lem}\label{lem4}$C[(K,p,q,l), E] = C[(K,p+l+1,q+l+1, -1), E]$
\end{lem}

\begin{proof}
We will just give an idea of the proof as there are too many terms to provide a closed formula (a similar technique is used in \cite{knotsinlattice}, Example 7.9). Let $(K,p,q,l)$ be a characteristic cohomology class.  We want to compute $\partial \circ H$ and $H\circ \partial$ to see when they cancel.  For $\partial_e$ and $\partial_u$ with $u \neq v$ and $u\neq w$, the case is the same as in \cite{knotsinlattice} Example 7.9.  As for $\partial_v\circ H$ and $H\circ \partial_v$, they are sums of terms of the form $U^s[L,(E\cup e)-v]$. Since $H$ increases the Maslov grading by exactly 1, two identical generators must always have same $U$ exponent.  Moreover, the contributions of terms of the form $[L,E-v]$ and $[L,E]$ cancel if and only if they are of same type (similarly for $[L,E-w]$ and $[L,E]$). Notice $[L,H]$ and $[L+2ne^*, H]$ are always of same type. Using the fact that if $n:=\frac{l+1}{2}$, then $(K,p,q,l)+2ne^* = (K,p+l+1,q+l+1, -1)$, we get
\[
C[(K,p,q,l),E] = C[(K,p+l+1,q+l+1,-1),E]
\]
\end{proof}

We define the blow-down map: 
\[
\begin{array}{cccl}
S : & \mathbb{CF^-}(G') & \to & \mathbb{CF^-}(G) \\
& [(K,p,q,l), E] & \mapsto & 
\begin{cases}
0 & \text{if } e \in E\\
U^s[(K,p+l,q+l), E] & \text{if } e \notin E
\end{cases}
\end{array}
\]
where
\[
s := \frac{l^2-1}{2} + g[(K,p+l,q+l), E] - g[(K,p,q,l)E]
\]

\begin{lem}\label{lem5}$s$ is always positive.
\end{lem}

\begin{proof}
Suppose $I\subset E$ such that $v,w \notin I$ and $e\notin E$. We have
\begin{align*}
f[(K,p+l,q+l),I] &= f[(K,p,q,l),I]\\
f[(K,p+l,q+l),I\cup v] &= f[(K,p,q,l),I\cup v] + \frac{l+1}{2}\\
f[(K,p+l,q+l),I\cup w] &= f[(K,p,q,l),I\cup w] + \frac{l+1}{2}\\
f[(K,p+l,q+l),I\cup v\cup w] &= f[(K,p,q,l),I\cup v\cup w] + l+2\\
\end{align*}
Therefore, 
\[
g[(K,p+l,q+l),E] - g[(K,p,q,l),E] \geqslant \min\{0,l+2,\frac{l+1}{2}\}
\]
And if $l$ is odd,
\[
\frac{l^2-1}{8} + \mathrm{min}\{0,l+2,\frac{l+1}{2}\} \geqslant 0
\]
\end{proof}
\begin{rem}Lemma \ref{lem5} guarantees that $S$ is well defined.
\end{rem}

\begin{lem}\label{lem6}
$S$ respects the Maslov grading
\end{lem}

\begin{proof}
Recall that $gr[K,E] = 2g[K,E] + \lvert E\rvert + \frac{1}{4}(K^2 + \lvert Vert(G)\rvert)$ and notice that 
\begin{equation}\label{eq6}
(K,p,q,l)^2 = (K,p+l,q+l)^2 -l^2
\end{equation}
If $e\in E$ then $S[(K,p,q,l),E] = 0$. Suppose now that $e\notin E$. We have: 
\begin{align*}
gr(S[(K,p,q,l),E] &= gr(U^s[(K,p+l,q+l),E])\\
&= -2s + 2g[(K,p+l,q+l),E] + \lvert E\rvert + \frac{1}{4}((K,p+l,q+l)^2 + \lvert Vert(G)\rvert)\\
&=2g[(K,p,q,l),E] + \lvert E\rvert + \frac{1}{4}((K,p,q,l)^2 + \lvert Vert(G')\rvert + l^2 -1) - \frac{l^2-1}{4}\\
&= gr[(K,p,q,l),E]
\end{align*}
\end{proof}

\begin{lem}\label{lem7}
$S$ is a chain map.
\end{lem}

\begin{proof}
\begin{itemize}
\item If $e\in E$, then $\partial \circ S[(K,p,q,l),E] = 0$. Using the fact that $(K,p,q,l)+2e^* = (K,p+2,q+2,l-2)$, we have:
\begin{align*}
S\circ \partial'[(K,p,q,l), E] &= U^xS[(K,p,q,l),E-e] + U^yS[(K,p+2,q+2,l-2), E-e]\\
&= U^x[(K,p+l,q+l),E-e] + U^y[(K,p+l,q+l),E-e]
\end{align*}
and since $S$ preserves Maslov gradings, $x = y$. Therefore $S\circ \partial' = \partial \circ S = 0$.
\item If $e\notin E$ and $v,w\notin E$, then using the fact that for $u\neq v$ and $u\neq w$, \[S[(K,p,q,l) + 2u^*, E-u] = U^{s'}[(K,p+l,q+l)+2u^*, E-u]\]and using the fact that $S$ preserves Maslov gradings, we have:
\begin{align*}
\partial \circ S[(K,p,q,l), E] &= U^s\partial [(K,p+l,q+l),E]\\
&= \sum_{u\in E}U^{s+x_{u}}[(K,p+l,q+l),E-u] + U^{s+y_{u}}[(K,p+l,q+l)+2u^*, E-u]\\
&= S\left( \sum_{u\in E}U^{x'_{u}}[(K,p,q,l),E-u] + U^{y'_{u}}[K,p,q,l)+2u^*, E-u]\right)\\
&= S\circ \partial[(K,p,q,l),E]
\end{align*}
\item If $e\notin E$ and $v\in E$ (or $w\in E$), we can look at the $\partial_v$ (or $\partial_w$) term:
\begin{align*}
\partial_v\circ S[(K,p,q,l),E] &= U^x[(K,p+l,q+l),E-v] + U^y[(K,p+l,q+l)+2v^*,E-v]\\
S\circ \partial'_v[(K,p,q,l),E] &= U^{x'}[(K,p+l,q+l),E-v] + U^{y'}S[(K,p,q,l)+2v^*,E-v]
\end{align*}
Moreover,
\begin{align*}
(K,p,q,l) +2v^* &= (K', p+ 2m_v -2, q, l+2)\\
(K,p+l,q+l) + 2v^* &= (K', p+l+2m_v, q+l+2)
\end{align*}
Since $S$ preserves Maslov gradings, $x = x'$ and 
\[U^{y'}S[(K,p,q,l)+2v^*,E-v] = U^y[(K,p+l,q+l)+2v^*,E-v]\]
A similar argument holds for $w\in E$.
\end{itemize}
\end{proof}

We define the blow-up map: 
\[
\begin{array}{cccl}
T: & \mathbb{CF^-}(G) & \to & \mathbb{CF^-}(G')\\
& [(K,p,q), E] & \mapsto & C[(K,p+1,q+1,-1),E]
\end{array}
\]

\begin{lem}\label{lem8}
$T$ preserves the Maslov grading.
\end{lem}

\begin{proof}
Thanks to Lemma \ref{lem3} we only need to consider the case where $e\notin E$. Lemma \ref{lem4} tells us that 
\[
T[(K,p,q), E] = C[(K,p+1,q+1,-1), E] = C[(K,p-1,q-1,1),E]
\]
and since $C$ preserves Maslov grading, 
\[
gr[(K,p+1,q+1,-1),E] = gr[(K,p-1,q-1,1),E]
\]
Moreover, equation (\ref{eq6}) tells us that
\[
(K,p+1,q+1,-1)^2 = (K,p,q)^2-1 = (K,p-1,q-1,1)^2
\]
Therefore, 
\[
g[(K,p+1,q+1,-1),E] = g[(K,p-1,q-1,1),E]
\]
If $e,v,w\notin I$, then
\begin{align*}
f[(K,p+1,q+1,-1),I] &= f[(K,p-1,q-1,1),I]\\
f[(K,p+1,q+1,-1),I\cup v] &= f[(K,p-1,q-1,1),I\cup v] + 1\\
f[(K,p+1,q+1,-1),I\cup w] &= f[(K,p-1,q-1,1),I\cup w] + 1\\
f[(K,p+1,q+1,-1),I\cup v\cup w] &= f[(K,p-1,q-1,1),I\cup v\cup w] + 2\\
\end{align*}
Therefore, 
\[
g[(K,p+1,q+1,-1),E] = f[(K,p+1,q+1, -1),I] \Rightarrow v,w\notin I
\]
Since 
\[
v,w\notin I \Rightarrow f[(K,p+1,q+1,-1),I] = f[(K,p,q),I]
\]
we now know that $g[(K,p+1,q+1,-1),E] \geqslant g[(K,p,q),E]$. Moreover the proof of Lemma \ref{lem5} shows that
\[
g[(K,p+1,q+1,-1),E] \leqslant g[(K,p,q),E]
\]
Therefore $T$ preserves the Maslov grading
\end{proof}

\begin{lem}\label{lem9}$T$ is a chain map.
\end{lem}

\begin{proof}
Consider $[(K,p,q),E]$ a generator of $\mathbb{CF}^-(G)$.
\begin{itemize}
\item If $v, w \notin E$
\[
T\circ \partial[(K,p,q),E] = T\left(\sum_{u\in E}U^x[(K,p,q),E-u] + U^y[(K,p,q)+2u^*, E-u]\right)
\]
And since $u\neq v$ and $u\neq w$,
\[
T[(K,p,q)+2u^*,E-u] = C[(K,p+1,q+1)+2u^*,E-u]
\]
Therefore, since $T$ preserves Maslov gradings, 
\[
T\circ \partial[(K,p,q),E] = \partial'\circ T[(K,p,q),E]
\]
\item Now consider the $v$ term:
\begin{align*}
(K,p,q) + 2v^* &= (K', p+2m_v, q+2)\\
(K,p+1,q+1,-1) + 2v^* &= (K', p+1+2m_v-2, q+1,1)
\end{align*}
Therefore, using Lemma \ref{lem4},
\begin{align*}
T[(K,p,q)+2v^*, E-v] &= C[(K',p+2m_v +1, q+3,-1),E-v]\\
&= C[(K',p+2m_v -1, q+1, 1), E-v]\\
&= C[(K,p+1,q+1,-1) +2v^*, E-v]
\end{align*}
Since $T$ preserves Maslov gradings, this guarantees that
\[
\partial'_v \circ T = T\circ \partial_v 
\]
A similar argument holds for the $w$ term.
\end{itemize}
\end{proof}

\begin{lem}\label{lem10}
$H_0$ is filtered.
\end{lem}

\begin{proof}
Let $[K,E]$ be a generator of $\mathbb{CF^-}(G')$ with $e\notin E$.
\begin{itemize}
\item If $K(e)\geqslant 1$ then $K(e) + 2v_0.e = K(e) \geqslant 1$. By \cite{knotsinlattice} Lemma 7.3,
\[
a_e[K,E\cup e] = a_e[K+2v_0^*, E\cup e] = 0
\]
Therefore, by \cite{knotsinlattice} formula (3.2),
\[
A'[K,E] - A'[K,E\cup e] =0
\]
\item If $K(e) \leqslant -5$ then by Lemma \ref{lem1}, 
\[
b_e[K-2e^*, E\cup e] = b_e[K-2e^*+2v_0^*, E\cup e] = 0
\]
Therefore, by \cite{knotsinlattice} formula (3.3), 
\[
A'[K-2e^*,E\cup e] - A'[K,E] =0
\]
\item If $K(e) = -1$ and $[K,E]$ is of type-a, then since $v_0$ is a leaf,
\begin{equation}\label{eq7}
g[K,E] = g[K,E\cup e] \leqslant g[K+2v_0^*, E\cup e] \leqslant g[K+2v_0^*, E] \leqslant g[K,E] + 1
\end{equation}
\begin{itemize}
\item Suppose there exists $I_0\subset E$ such that $g[K,E] = f[K,I_0]$ and\\ $v_0.\sum_{u\in I_0}u = 0$. Then 
\[
f[K+2v_0^*, I_0] = f[K,I_0] = g[K,E]
\]
Therefore, (\ref{eq7}) implies that $g[K+2v_0^*,E] = g[K+2v_0^*,E\cup e]$. Using \cite{knotsinlattice} formula (3.2) we see that $A'[K,E] - A'[K,E\cup e] =0$.
\item Suppose that for any $I\subset E$
\[
g[K,E] = f[K,I] \Rightarrow v_0.\sum_{u\in I}u = 1
\]
If there is an $I\subset E$ such that $f[K+2v_0^*,I] = g[K+2v_0^*, E\cup e]$ then obviously $a_e[K+2v_0^*, E\cup e] =0$. If for every $I\subset E$ satisfying $g[K,E] = f[K,I]$,  $g[K+2v_0^*, E\cup e] = f[K+2v_0^*, I] -1$, then consider $J\subset E\cup e$ such that $g[K+2v_0^*, E\cup e] = f[K+2v_0^*, J]$. Necessarily $ f[K+2v_0^*, J] \geqslant  f[K, J]$, therefore $f[K,J] = f[K,I]$ by minimality of $f[K,I]$. This implies that $v_0.\sum_{u\in J}u = 1$, so $f[K+2v_0^*, J] = f[K,J] + 1$ which is absurd.
\end{itemize}
\item If $K(e) = -3$ and $[K,E]$ is of type-b, then Lemma 3.1 tells us that $b_e[K-2e^*, E\cup e] = 0$. For any $I\subset E$ with $v,w\notin I$,
\begin{align*}
f[K-2e^*+2v_0^*, I\cup e] &= f[K-2e^*+2v_0^*, I] -1\\
f[K-2e^*+2v_0^*, I\cup v\cup e] &= f[K-2e^*+2v_0^*, I\cup v]\\
f[K-2e^*+2v_0^*, I\cup w\cup e] &= f[K-2e^*+2v_0,^* I\cup w]\\
f[K-2e^*+2v_0^*, I\cup v\cup w\cup e] &= f[K-2e^*+2v_0^*, I\cup v\cup w]+ 1\\
\end{align*}
Therefore the only problematic case is when $\forall I \subset E\cup e$,
\begin{equation}\label{eq8}
g[K-2e^*+2v_0^*, E\cup e] = f[K-2e^*+2v_0^*, I] \Rightarrow v,w\in I
\end{equation}
Suppose it is the case. We then know that $g[K-2e^*+2v_0^*, E\cup e] = g[K-2e^*+2v_0^*, E]$. Since $g[K-2e^*,E\cup e] < g[K-2e^*, E]$, we know that there exists $I_0\subset E$ with $v,w\notin I_0$ such that $g[K-2e^*,E\cup e] = f[K-2e^*, I_0\cup e]$. Therefore,
\[
\begin{array}{ccll}
f[K-2e^*+2v_0^*, I_0\cup e] &\leqslant &f[K-2e^*, I_0\cup e] +1\\
& &= g[K-2e^*, E\cup e] +1\\
& &= g[K-2e^*, E]\\
&&\leqslant g[K-2e^*+2v_0^*, E] = g[K-2e^*+2v_0^*, E\cup e]
\end{array}
\]
which contradicts condition (\ref{eq8}).
\end{itemize}
\end{proof}
This implies that $C$ and $H$ are also filtered.

\begin{lem}\label{lem11}$S$ is filtered.
\end{lem}
\begin{proof}
Suppose $e\notin E$. 
\begin{align*}
A'[(K,p,q,l),E]-&A(U^s[(K,p+l,q+l),E]) \\
&= s + \frac{1}{2}\left(L_{[(K,p,q,l),E]}(\Sigma') - L_{[(K,p+l,q+l),E]}(\Sigma) \right)\\
&= \frac{l^2-1}{8} + g[(K,p+l,q+l)+2v_0^*, E] - g[(K,p,q,l)+2v_0^*, E]
\end{align*}
And similarly to the proof of Lemma \ref{lem5},
\[
g[(K,p+l,q+l)+2v_0^*, E] - g[(K,p,q,l)+2v_0^*, E] \geqslant \min\{0,l+2,\frac{l+1}{2}\}
\]
Therefore, 
\[
A'[(K,p,q,l),E]-A(U^s[(K,p+l,q+l),E]) \geqslant 0
\]
\end{proof}

\begin{lem}\label{lem12}$T$ is filtered.
\end{lem}

\begin{proof}
If $e\in E$ then $T[K,E] = 0$. Suppose $e\notin E$. Since $C$ is filtered, it suffices to show
\[
(*) := A'[(K,p+1,q+1,-1),E] - A[(K,p,q),E] \leqslant 0
\]
Moreover, the proof of Lemma \ref{lem8} tells us that
\[
g[(K,p+1,q+1,-1),E] = g[(K,p,q),E]
\]
Therefore, 
\[
S[(K,p+1,q+1,-1),E] = U^0[(K,p,q),E]
\]
and as a consequence,
\[
(*) = A'[(K,p+1,q+1,-1),E] - A(S[(K,p+1,q+1,-1),E])
\]
Using the proof of Lemma \ref{lem11}, 
\[
(*) = 0 + g[(K,p,q)+2v_0^*, E] - g[(K,p+1,q+1,-1)+2v_0^*, E]
\]
If $v,w \notin I\subset E$, then
\begin{align*}
f[(K,p,q) +2v_0^*, I] &= f[(K,p+1,q+1, -1)+2v_0^*, I]\\
f[(K,p,q) +2v_0^*, I\cup v] &= f[(K,p+1,q+1, -1)+2v_0^*, I\cup v]\\
f[(K,p,q) +2v_0^*, I\cup w] &= f[(K,p+1,q+1, -1)+2v_0^*, I\cup w]\\
f[(K,p,q) +2v_0^*, I\cup v\cup w] &= f[(K,p+1,q+1, -1)+2v_0^*, I\cup v\cup w] + 1\\
\end{align*}
Therefore $(*) \leqslant 0$ unless $\forall I \subset E$,
\begin{equation}\label{eq9}
g[(K,p+1,q+1,-1)+2v_0^*, E] = f[(K,p+1,q+1,-1) + 2v_0^*, I] \Rightarrow v,w\in I
\end{equation}
From the proof of Lemma \ref{lem8}, we know that
\begin{equation}\label{eq10}
g[(K,p+1,q+1,-1),E] = f[(K,p+1,q+1, -1),I] \Rightarrow v,w\notin I
\end{equation}
Suppose $I$ satisfies (\ref{eq9}) and $J$ satisfies (\ref{eq10}). We have
\[
f[(K,p+1,q+1,-1),J] \leqslant f[(K,p+1,q+1,-1),I] \leqslant f[(K,p+1,q+1,-1)+2v_0^*,I]
\]
and
\[
f[(K,p+1,q+1,-1)+2v_0^*,I] \leqslant f[(K,p+1,q+1,-1)+2v_0^*,J] \leqslant f[(K,p+1,q+1,-1),J] +1
\]
Therefore either
\[
f[(K,p+1,q+1,-1)+2v_0^*,I] = f[(K,p+1,q+1,-1)+2v_0^*,J]
\]
or
\[
f[(K,p+1,q+1,-1),I] = f[(K,p+1,q+1,-1),J]
\]
Both cases are absurd.
\end{proof}

\begin{prop}$(\mathbb{CF^-}(G),A)$ and $(\mathbb{CF^-}(G'), A')$ are filtered graded chain homotopic.
\end{prop}
\begin{proof}
$S$ and $T$ are filtered graded chain maps. Notice that $S\circ T = \mathrm{Id_{\mathbb{CF^-}(G)}}$ (because $S$ and $T$ preserve Maslov gradings). Let's show that $T\circ S = C$. 
\begin{itemize}
\item If $e\in E$ then $T\circ S[(K,p,q,l),E] = 0 = C[(K,p,q,l),E]$. 
\item If $e\notin E$, then 
\begin{align*}
T\circ S[(K,p,q,l),E] &= U^sC[(K,p+l+1,q+l+1,-1),E]\\
&= U^sC[(K,p,q,l),E]
\end{align*}
and since $T$, $S$ and $C$ preserve Maslov gradings, $s = 0$.
\end{itemize}
Therefore $T\circ S = \mathrm{Id}_{\mathbb{CF^-}(G')}+ \partial\circ H + H\circ \partial$ and $H$ is filtered.
\end{proof}

\section{Conclusion}

Ozsváth,  Stipsicz and Szabó show that the other types of blow-ups and blow-downs preserve the filtered lattice chain homotopy type (\cite{knotsinlattice} Corollary 4.9 and Theorem 7.13). We deduce Theorem \ref{thm1} from this and from the previous sections.

\end{document}